\documentclass[a4paper,11pt]{article}

\usepackage{amsmath}
\usepackage{amssymb}
\usepackage{amsthm}
\usepackage{graphicx}
\usepackage{amscd}
\usepackage{epic, eepic}
\usepackage{url}
\usepackage{color}
\usepackage[utf8]{inputenc} 
\usepackage{comment}
\usepackage{enumerate}   
\usepackage{todonotes}
\usepackage{graphicx}
\usepackage{epstopdf}
\usepackage{enumitem}
\usepackage{tikz}
\usepackage[top=24mm, bottom=25mm, left=25mm, right=25mm]{geometry}
\usepackage[bookmarks=false,hidelinks]{hyperref}
\usepackage{mathrsfs}

\makeatletter
\renewenvironment{proof}[1][\proofname] {\par\pushQED{\qed}\normalfont\topsep6\p@\@plus6\p@\relax\trivlist\item[\hskip\labelsep\bfseries#1\@addpunct{.}]\ignorespaces}{\popQED\endtrivlist\@endpefalse}
\makeatother

\newtheorem{theorem}{\bf Theorem}[section]
\newtheorem{lemma}[theorem]{\bf Lemma}
\newtheorem{corollary}[theorem]{\bf Corollary}

\theoremstyle{definition}

\newtheorem{remark}[theorem]{\bf Remark}
\newtheorem{definition}[theorem]{\bf Definition}

\def\mc{\mathrm{mc}}
\def\sp{\mathrm{sp}}
\def\ex{\mathrm{ex}}
\newcommand{\vx}{\mathbf{x}}
\newcommand{\vy}{\mathbf{y}}
\newcommand{\C}{\mathscr{C}}
\newcommand{\N}{\mathbb{N}}
\newcommand{\R}{\mathbb{R}}
\newcommand{\tr}{\text{tr}}

\newcommand{\inprod}[2]{\langle #1, #2 \rangle}

\date{}
\title{On MaxCut and the Lov\'asz theta function}

\author{Igor Balla\thanks{Einstein Institute of Mathematics, Hebrew University of Jerusalem, Israel. Email: \textbf{iballa1990@gmail.com}.}
\and Oliver Janzer\thanks{Department of Pure Mathematics and Mathematical Statistics, University of Cambridge, United Kingdom. Research supported by a fellowship at Trinity College. Email: \textbf{oj224@cam.ac.uk}.}
\and Benny Sudakov\thanks{Department of Mathematics, ETH Z\"urich, Switzerland. Research supported in part by SNSF grant 200021\_196965. Email: \textbf{benjamin.sudakov@math.ethz.ch}.}}

\begin{document}
	
\maketitle

\begin{abstract}
	In this short note we prove a lower bound for the MaxCut of a graph in terms of the Lov\'asz theta function of its complement. We combine this with known bounds on the Lov\'asz theta function of complements of $H$-free graphs to recover many known results on the MaxCut of $H$-free graphs. In particular, we give a new, very short proof of a conjecture of Alon, Krivelevich and Sudakov about the MaxCut of graphs with no cycles of length $r$.
\end{abstract}

\section{Introduction}

The aim of this note is to establish a connection between two very important graph parameters: the MaxCut and the Lov\'asz theta function. The MaxCut of a graph $G$, denoted by $\mathrm{mc}(G)$, is the size of the largest bipartite subgraph in $G$. This parameter has been extensively studied in the last 50 years both in Theoretical Computer Science and in Extremal Graph Theory. One important direction of research has been the study of the smallest possible MaxCut that a graph satisfying some given conditions can have. A very simple probabilistic argument shows that for any graph $G$ with $m$ edges, we have $\mc(G)\geq m/2$. It is therefore natural to consider the so-called \emph{surplus} of $G$, defined as $\sp(G)=\mc(G)-e(G)/2$.

By a classical result of Edwards \cite{Edw73,Edw75}, for every graph $G$ with $m$ edges, we have $\sp(G)\geq \frac{\sqrt{8m+1}-1}{8}$, and this is tight whenever $G$ is a complete graph on an odd number of vertices. With the aim of improving this bound for graphs that are ``far from complete", in the 70s Erd\H os and Lov\'asz \cite{Erd79} initiated the study of the surplus in $H$-free graphs. Given a positive integer $m$ and graph $H$, let $\sp(m,H)$ denote the smallest possible value of $\sp(G)$ for an $H$-free graph $G$ with $m$ edges.

Alon, Bollob\'as, Krivelevich and Sudakov \cite{ABKS03} showed that for every fixed graph $H$ there exist positive constants $\varepsilon=\varepsilon(H)$ and $c=c(H)$ such that $\sp(m,H)\geq cm^{1/2+\varepsilon}$ for all $m$. One of the main conjectures in the area asserts that in fact there are suitable positive constants $c=c(H)$, $\varepsilon=\varepsilon(H)$ such that $\sp(m,H)\geq cm^{3/4+\varepsilon}$, but it is not even known whether there is an absolute constant $\alpha>1/2$ such that for every $H$ we have $\sp(m,H)\geq cm^{\alpha}$, for some $c=c(H)>0$.

Another important direction of research is to determine, for various graphs $H$, the order of magnitude of $\sp(m,H)$. The first graph that was considered is the triangle; it was shown by Erd\H os and Lov\'asz \cite{Erd79} that $\sp(m,K_3)=\Omega(m^{2/3}(\log m/\log \log m)^{1/3})$. This bound was improved by Shearer \cite{She92} to $\sp(m,K_3)=\Omega(m^{3/4})$, and finally Alon \cite{Alon94,Alon96} proved that $\sp(m,K_3)=\Theta(m^{4/5})$. More generally, the surplus in graphs without short cycles has been extensively studied, starting with the work of Erd\H os and Lov\'asz. Alon, Bollob\'as, Krivelevich and Sudakov \cite{ABKS03} showed that any $m$-vertex graph  with girth at least $r+1$ (i.e., not containing a cycle of length at most $r$) has surplus $\Omega_r(m^{(r+1)/(r+2)})$. This was strengthened for even values of $r$ by Alon, Krivelevich and Sudakov \cite{AKS05}, who proved that $\sp(m,C_r)=\Omega_r(m^{(r+1)/(r+2)})$ whenever $r$ is even. They conjectured that the same bound holds for odd values of $r$ as well. Zeng and Hou \cite{ZH18} proved the weaker bound $\sp(m,C_r)\geq m^{(r+1)/(r+3)+o(1)}$, and Fox, Himwich and Mani \cite{FHM23} improved this to $\sp(m,H)=\Omega_r(m^{(r+5)/(r+7)})$. Finally, the conjecture of Alon, Krivelevich and Sudakov was proved by Glock, Janzer and Sudakov \cite{GJS21}. Combined with earlier constructions of Alon and Kahale \cite{AK98}, this showed that $\sp(m,C_r)=\Theta_r(m^{(r+1)/(r+2)})$ holds for all odd $r$. One of the purposes of this note is to provide a novel, short proof of the result of Glock, Janzer and Sudakov. We note that this proof works identically also in the case where $r$ is even.

\begin{theorem} \label{thm:odd cycles}
	For each $r\geq 3$, we have $\sp(m,C_r)=\Omega_r(m^{(r+1)/(r+2)})$.
\end{theorem}

We will also use our method to give a new proof of the following results of Alon, Krivelevich and Sudakov\footnote{The latter of these two results is not explicitly stated in \cite{AKS05}, but can be obtained straightforwardly using their method.}.

\begin{theorem} \label{thm:vxplusforest}
    If $H$ is a graph that has a vertex whose removal makes the graph acyclic, then $\sp(m,H)=\Omega_H(m^{4/5})$.
\end{theorem}

\begin{theorem} \label{thm:vxplusforestturan}
    If $H$ is a graph that has a vertex whose removal makes the graph acyclic and its Tur\'an number satisfies $\ex(n,H)=O_{H}(n^{1+\alpha})$ for some constant $\alpha$, then $\sp(m,H)=\Omega_{H}(m^{\frac{\alpha+2}{2\alpha+2}})$.
\end{theorem}

\noindent We remark that Theorem \ref{thm:vxplusforest} is tight for the triangle and Theorem \ref{thm:vxplusforestturan} is tight for $K_{2,t}$ \cite{AKS05}. Theorem \ref{thm:vxplusforestturan} is tight also for the cycles of length $4$, $6$ and $10$ (see \cite{AKS05}), and it is possibly tight for all even cycles, but corresponding extremal constructions are not known even for the Tur\'an number.
The three theorems above together recover many known results on the MaxCut of $H$-free graphs (and most of those where the right exponent is known), with a unified approach and short proofs in each case.

We will derive these theorems from a new general lower bound on the surplus of an arbitrary graph. This bound involves the Lov\'asz theta function of the complement of the graph. Lov\'asz \cite{Lov79} introduced this important graph parameter in order to bound the Shannon capacity of a graph. One reason why the Lov\'asz theta function is useful is that it has many equivalent definitions. Many of them have already appeared in Lov\'asz's work, and several others have been found since. The definition that will be the most useful for us is as follows. (Note that it is well-defined as if the vectors $\vx_v$ correspond to the vertices of a regular simplex whose centre is in the origin, then all $\langle \vx_u,\vx_v\rangle$ are the same and negative, showing that there exists some $\kappa\geq 2$ for which $\langle \vx_u,\vx_v\rangle=-\frac{1}{\kappa-1}$ for all $u\neq v$.)

\begin{definition} \label{def:strict vector chromatic}
    The Lov\'asz theta function of a graph $G$, denoted $\vartheta(G)$, is the minimal $\kappa \geq 2$ for which there exists a unit vector $\vx_v$ (in some Euclidean space) for each vertex $v$ such that $\langle \vx_u,\vx_v\rangle=-\frac{1}{\kappa-1}$ holds whenever $u$ and $v$ are distinct vertices in $G$ and $uv\not \in E(G)$.
\end{definition}

\noindent The aforementioned definition was actually introduced by Karger, Motwani, and Sudan \cite{KMS98} as the so-called strict vector chromatic number of the complement of $G$, who showed that it is equivalent to $\vartheta(G)$. See Knuth \cite{Knuth94} for a survey on the Lov\'asz theta function and its applications.

We are now ready to state the main result of this paper.

\begin{theorem} \label{thm:MaxCut bound with theta}
    Let $G$ be a graph which has $m$ edges and let $\vartheta = \vartheta \left( \bar{G} \right)$ be the Lov\'asz theta function of the complement graph $\bar{G}$. Then
    $$\sp(G)\geq \frac{1}{\pi}\cdot \frac{m}{\vartheta-1}.$$
\end{theorem}

Theorem \ref{thm:MaxCut bound with theta} will in fact be deduced from a slightly stronger result in which we replace $\vartheta$ by the vector chromatic number, also defined by Karger, Motwani, and Sudan \cite{KMS98}.

\begin{definition}
    The vector chromatic number of a graph $G$, denoted $\chi_{\textrm{vec}}(G)$, is the minimal $\kappa \geq 2$ for which there exists a unit vector $\vx_v$ (in some Euclidean space) for each vertex $v$ such that $\langle \vx_u,\vx_v\rangle\leq -\frac{1}{\kappa-1}$ holds whenever $u$ and $v$ are distinct vertices in $G$ and $uv \in E(G)$.
\end{definition}

Note that it is trivial from the definitions that for any graph $G$ we have $\vartheta\left(\bar{G}\right)\geq \chi_{\textrm{vec}}(G)$. Hence, the following result strengthens Theorem \ref{thm:MaxCut bound with theta}.

\begin{theorem} \label{thm:MaxCut bound with vecchrom}
    Let $G$ be a graph which has $m$ edges. Then
    $$\sp(G)\geq \frac{1}{\pi}\cdot \frac{m}{\chi_{\textrm{vec}}(G)-1}.$$
\end{theorem}

\section{The proof of Theorem \ref{thm:MaxCut bound with vecchrom}}

In this section we prove Theorem \ref{thm:MaxCut bound with vecchrom}. In the proof we will use the semidefinite programming technique. This method was first utilised in the algorithmic context of the MaxCut problem by Goemans and Williamson \cite{GW95}. Carlson, Kolla, Li, Mani, Sudakov and Trevisan \cite{CKLMST21} were the first ones to apply the technique to bound the MaxCut of $H$-free graphs, and Glock, Janzer and Sudakov \cite{GJS21} built on their approach to obtain several tight results in this context. The idea is to assign vectors $\vx_v$ (in some Euclidean space $\mathbb{R}^N$) to the vertices of the given graph $G$ in a way that for a typical edge $uv$ in $G$, the inner product $\langle \vx_u,\vx_v\rangle$ is negative. We then define a random cut as follows. Choose a unit vector $\textbf{z}\in \mathbb{R}^N$ uniformly at random and let $A=\{v\in V(G): \langle \vx_v,\textbf{z}\rangle\geq 0\}$ and $B=\{v\in V(G): \langle \vx_v,\textbf{z}\rangle < 0\}$. Now if an edge $uv$ in $G$ has $\langle \vx_u,\vx_v\rangle<0$, it belongs to the cut between $A$ and $B$ with probability more than $1/2$. Hence, by the choice of our vectors, we get a  random cut which on average contains more than half of the edges. This is made precise by the following lemma, which is essentially due to Goemans and Williamson \cite{GW95} (see \cite{GJS21} for a proof of this formulation of the result).

\begin{lemma} \label{lemma:general SDP}
    Let $G$ be a graph and let $N$ be a positive integer. Then, for any set of non-zero vectors $\{\vx_v:v\in V(G)\}\subset \mathbb{R}^N$, we have $\sp(G)\geq -\frac{1}{\pi}\sum_{uv\in E(G)} \arcsin\left(\frac{\langle \vx_u,\vx_v \rangle}{\|\vx_u\|\|\vx_v\|}\right)$.
\end{lemma}

We will use the following straightforward corollary of Lemma \ref{lemma:general SDP}.

\begin{corollary} \label{cor:SDP}
    Let $G$ be a graph and let $N$ be a positive integer. Then, for any set of unit vectors $\{\vx_v:v\in V(G)\}\subset \mathbb{R}^N$ such that $\langle \vx_u,\vx_v\rangle\leq 0$ whenever $uv\in E(G)$, we have $\sp(G)\geq -\frac{1}{\pi}\sum_{uv\in E(G)} \langle \vx_u,\vx_v \rangle$.
\end{corollary}

\begin{proof}
    By Lemma \ref{lemma:general SDP} and since $\vx_v$ is a unit vector for each $v\in V(G)$, we have $\sp(G)\geq -\frac{1}{\pi}\sum_{uv\in E(G)} \arcsin(\langle \vx_u,\vx_v \rangle)$. Recall that $\sin(x)\geq x$ for all $x\leq 0$ and therefore $\arcsin(x) \leq x$. Since $\langle \vx_u,\vx_v\rangle\leq 0$ whenever $uv\in E(G)$, the result follows.
\end{proof}

It is now very simple to prove Theorem \ref{thm:MaxCut bound with vecchrom}. Indeed, given a graph $G$, we may choose a unit vector $\vx_v$ for each vertex $v\in V(G)$ such that $\langle \vx_u,\vx_v\rangle\leq -\frac{1}{\chi_{\textrm{vec}}(G)-1}$ holds whenever $uv\in E(G)$. Applying Corollary \ref{cor:SDP} with these vectors, the desired inequality follows.

\section{Lov\'asz theta function versus the number of edges}

In this section we prove Theorems \ref{thm:odd cycles}, \ref{thm:vxplusforest} and \ref{thm:vxplusforestturan}. Using Theorem \ref{thm:MaxCut bound with theta}, these results can be reduced to finding suitable upper bounds for the Lov\'asz theta function of complements of $H$-free graphs $G$ (for suitable $H$) in terms of the number of edges of $G$. The following functions will therefore be of great importance in our proofs.

\begin{definition}
For a graph $H$ and $n, m \in \N$, let $\lambda(n, H)$ and $\mu(m,H)$ denote the maximum of $\vartheta\left(\bar{G}\right)$ over all $H$-free graphs $G$ with $n$ vertices and $m$ edges, respectively.
\end{definition}

$\lambda(n,H)$ was studied extensively in \cite{BLS20} and we will use the results therein as a black box to obtain bounds on $\mu(m,H)$, via a lemma which relates these two functions. Before we state this result, we need a little preparation. The next lemma gives two further (equivalent) definitions of the Lov\'asz theta function.

\begin{lemma}[\cite{Lov79}] \label{lem:Lovasz Gram}
The Lov\'asz theta function of a graph $G$ is equal to both of the following quantities: 
\begin{enumerate}
    \item the maximum of the largest eigenvalue of the Gram matrix $M$ defined by $M_{u,v} = \inprod{\vx_u}{\vx_v}$ for $u,v \in V(G)$, over all collections of unit vectors $\{\vx_v : v \in V(G) \}$ (in some Euclidean space) satisfying $\langle \vx_u,\vx_v\rangle = 0$ for all $uv \in E(G)$,
    \item the maximum of $\sum_{v \in V(G)}{\inprod{\vx}{\vx_v}^2}$ over all collections of unit vectors $\{\vx\}\cup \{\vx_v : v \in V(G) \}$ (in some Euclidean space) satisfying $\langle \vx_u,\vx_v\rangle = 0$ for all $uv \in E(G)$. 
\end{enumerate}
\end{lemma}
%\begin{remark}
%Since the largest eigenvalue of a matrix is the maximum of its Rayleigh quotient, it is not hard to see that the quantities in statements 1 and 2 of Lemma \ref{lem:Lovasz Gram} are equal.
%\end{remark}

For the following lemma, we will use the well known fact that the largest eigenvalue of a matrix $M$ is at most the maximum over all rows $r$ of the sum $\sum_{i}{|M_{r,i}|}$. 

\begin{lemma} \label{lemma:theta degree bound}
For any graph $G$ with maximum degree $\Delta$ we have 
\begin{enumerate}[label=(\alph*)]
    \item $\vartheta\left(\bar{G}\right) \leq \Delta+1$ and
    \item $\vartheta\left(\bar{G}\right)^2 \leq (1 + \Delta)\left(1+\max\left\{ \vartheta\left( \overline{G[N(v)]} \right) : v \in V(G) \right\}\right)$.
\end{enumerate} 

\end{lemma}
\begin{proof}
By statement 1 of Lemma~\ref{lem:Lovasz Gram}, there exists a collection $\C = \{\vx_v : v \in V(G) \}$ of unit vectors such that their Gram matrix $M$ has largest eigenvalue $\vartheta\left(\bar{G}\right)$ and $\inprod{\vx_u}{\vx_v} = 0$ for all $u \neq v$ such that $uv \notin E(G)$. Then we have
\begin{equation}
    \vartheta\left( \bar{G} \right) \leq \max_{v \in V(G)}{\sum_{u \in V(G)}{|\inprod{\vx_u}{\vx_v}|}}. \label{eqn:eigenvalue}
\end{equation}

Since $|\inprod{\vx_u}{\vx_v}| \leq 1$, we have that $\sum_{u \in V(G)}{|\inprod{\vx_u}{\vx_v}|} \leq 1 + \Delta$ for all $v \in V(G)$, so that claim (a) of Lemma \ref{lemma:theta degree bound} follows. Moreover, for all $v \in V(G)$ it follows from statement 2 of Lemma~\ref{lem:Lovasz Gram} that
$$\sum_{u \in V(G)}{\inprod{\vx_u}{\vx_v}^2}=1+\sum_{u\in N(v)} {\inprod{\vx_u}{\vx_v}^2}\leq 1+\max_{\|\vx\|=1} \sum_{u\in N(v)} {\inprod{\vx_u}{\vx}^2}\leq 1+\vartheta\left(\overline{G[N(v)]}\right)$$ and therefore using the Cauchy--Schwarz inequality we conclude $\left(\sum_{u \in V(G)}{|\inprod{\vx_u}{\vx_v}|}\right)^2 \leq (1 + \Delta) \left(1+\vartheta\left(\overline{G[N(v)]}\right)\right)$, so that claim (b) of Lemma \ref{lemma:theta degree bound} follows from equation (\ref{eqn:eigenvalue}).
\end{proof}

By partitioning the graph into vertices of high and low degree, we show in the following lemma that an upper bound on $\lambda(n, H)$ implies an upper bound on $\mu(m,H)$. We will need the well-known fact that $\vartheta\left(\bar{G}\right) \leq \chi(G)$ for all graphs $G$ (see e.g.\ Knuth \cite{Knuth94}) and the following simple lemma (see e.g.\ Corollaries 1.5.4 and 5.2.3 of Diestel \cite{D12}).

\begin{lemma} \label{lem:acyclic-free}
    If $H$ is acyclic, then any $H$-free $G$ has chromatic number $\chi(G) \leq |V(H)| - 1$. In particular, for any $H$-free graph $G$, we have $\vartheta(\bar{G})\leq |V(H)|-1$.
\end{lemma}

We are now ready to state and prove the lemma connecting $\lambda(n,H)$ to $\mu(m,H)$.

\begin{lemma} \label{lemma:vertex to edge}
Fix $\alpha \geq 1$ and let $H$ be a graph such that $\lambda(n , H) = O_H\left(n^{1/\alpha}\right)$. Then 
$$
\mu(m,H) = O_H \left(m^{1/(\alpha + 1)} \right)
$$    
and moreover, if $H$ has a vertex such that removing it makes $H$ acyclic then 
$$
\mu(m,H) = O_H \left(m^{1/(\alpha + 2)} \right).
$$
\end{lemma}
\begin{proof}
Let $d \in \R$ be a positive number to be chosen later and let $G$ be an $H$-free graph with $m$ edges. If we define $ S = \{v \in V(G) : d(v) \leq d\}$ and $T = V(G) \backslash S$, then it is easy to see from statement 2 of Lemma \ref{lem:Lovasz Gram} that
$$
\vartheta\left(\bar{G}\right) \leq \vartheta\left(\overline{G[S]}\right) + \vartheta\left(\overline{G[T]}\right).
$$
Now since $|T| d \leq \sum_{v \in T}{d(v)} \leq 2m$, we have that $|T| \leq 2m / d$ and thus 
$$
\vartheta\left(\overline{G[T]}\right) \leq \lambda(|T|, H) \leq 
O_H\left( \left(\frac{m}{d}\right)^{1/\alpha} \right).
$$
Also, it follows from claim (a) of Lemma \ref{lemma:theta degree bound} that $\vartheta\left(\overline{G[S]}\right) \leq d+1$, so letting $d = m^{1/(\alpha + 1)}$, we conclude the desired bound $\vartheta\left(\bar{G}\right) \leq d + O_H\left((m/d)^{1/\alpha}\right) = O_H\left(m^{1/(\alpha + 1)}\right)$.

Moreover, suppose that $H$ has a vertex such that the graph obtained by removing it, say $H'$, is acyclic. Then for any $v \in S$, the neighborhood of $v$ in $G[S]$ has no copy of $H'$ and hence, by Lemma \ref{lem:acyclic-free}, $\vartheta\left(\overline{G[S\cap N(v)]}\right) \leq |V(H')|-1=|V(H)|-2$. It follows from claim (b) of Lemma \ref{lemma:theta degree bound} that $\vartheta\left(\overline{G[S]}\right) \leq O_H\left( \sqrt{d} \right)$, so that letting $d = m^{2/(\alpha + 2)}$, we conclude $\vartheta\left(\bar{G}\right) \leq O_H\left( \sqrt{d} + (m/d)^{1/\alpha} \right) = O_H\left( m^{1/(\alpha + 2)} \right)$, as desired.
\end{proof}
\begin{remark}
Since any graph $G$ on $n$ vertices satisfies $\vartheta\left( \bar{G} \right) \leq \chi(G) \leq n$, the argument of Lemma \ref{lemma:vertex to edge} can also be used to show that any graph $G$ with $m$ edges satisfies $\vartheta\left(\bar{G} \right) \leq O\left(\sqrt{m}\right)$.  
\end{remark}

It was previously proved  in \cite{BLS20} that $\lambda(n, C_r)=O_r(n^{1/r})$ and $\lambda(n, H) = O_H(n^{\alpha/2})$ if $H$ is a graph that has a vertex whose removal makes $H$ acyclic\footnote{In fact, this was stated in \cite{BLS20} for graphs $H$ which have a vertex whose removal makes $H$ a tree, but the same proof works for this slightly more general class of graphs.} and its Tur\'an number satisfies $\ex(n, H) = O_H\left( n^{1 + \alpha} \right)$ for some constant $\alpha$. Using these bounds, Lemma \ref{lemma:vertex to edge} immediately yields the following corollary. 

\begin{corollary} \label{cor:theta vertex bounds}
For all $r\geq 3$, we have $\mu(m, C_r) = O_r\left(m^{1/(r+2)}\right)$. Moreover, if $H$ is a graph that has a vertex whose removal makes $H$ acyclic and its Tur\'an number satisfies $\ex(n, H) = O_H\left( n^{1 + \alpha} \right)$ for some constant $\alpha$, then $\mu(m, H) = O_{H}\left( m^{\alpha/(2\alpha + 2)} \right)$.
\end{corollary} 

\noindent Theorem \ref{thm:odd cycles} and Theorem \ref{thm:vxplusforestturan} now follow directly from Corollary \ref{cor:theta vertex bounds} and Theorem \ref{thm:MaxCut bound with theta}.

Additionally, for any graph $H$ with a vertex whose removal leaves it acyclic, we apply the argument used in \cite{BLS20} based on the trace method in order to obtain a bound on $\lambda(n, H)$ and therefore also $\mu(m, H)$. This bound is better than the one in Corollary \ref{cor:theta vertex bounds} whenever the Tur\'an exponent of $H$ is close to $2$, in particular whenever $H$ is not bipartite.

\begin{theorem} \label{thm:H lambda}
Let $H$ be a graph containing a vertex whose removal makes $H$ acyclic. Then
$$
\lambda(n, H) = O_H\left(n^{1/3}\right) \qquad \text{and} \qquad \mu(m, H) = O_H\left(m^{1/5}\right).
$$
\end{theorem}
\begin{proof}
Let $G$ be an $H$-free graph on $n$ vertices and let $\{\vx_v: v \in V(G)\}$ be a collection of unit vectors as in Lemma \ref{lem:Lovasz Gram}, chosen so that their Gram matrix given by $M_{u,v} = \inprod{\vx_u}{\vx_v}$ has largest eigenvalue $\vartheta\left(\bar{G}\right)$ and $\inprod{\vx_u}{\vx_v} = 0$ if $u \neq v$ and $uv \notin E(G)$. Note that $M$ is positive semidefinite, i.e.\ all eigenvalues of $M$ are nonnegative, and so we have
$$
\vartheta\left(\bar{G}\right)^3 \leq \tr(M^3) = \sum_{u, v, w \in V(G)}{\vx_u^\intercal \vx_v \vx_v^\intercal \vx_w \vx_w^\intercal \vx_u}
= \sum_{u \in V(G)}{ \vx_u^\intercal \left( \sum_{v \in N(u) \cup \{u\}}{\vx_v \vx_v^\intercal} \right)^2 \vx_u}.
$$
Now fix $u \in V(G)$ and note that since removing a vertex from $H$ leaves a subgraph which is acyclic, we must have $\chi(G[N(u)]) = O_H(1)$. Therefore, we can partition the neighborhood of $u$ into $t$ independent sets $B_1, \ldots, B_t$, where $t = O_H(1)$. Furthermore, define $B_0 = \{u\}$ and observe that for each independent set $B_i$, the corresponding vectors $\{ \vx_v : v \in B_i\}$ are orthogonal. By Parseval's identity, this implies that, setting $P_i = \sum_{v \in B_i}{\vx_v \vx_v^\intercal}$, we have $\vy^T P_i \vy=\sum_{v\in B_i} \langle \vy,\vx_v\rangle^2\leq \|\vy\|^2$ for any $\vy$, so $\lambda_1(P_i)\leq 1$. Hence, $P = \sum_{i=0}^t{P_i}$ satisfies $\lambda_1(P) \leq \sum_{i=0}^t{\lambda_1(P_i)} = O_H(1)$.  Moreover, $P$ is positive semidefinite, so $\lambda_1(P^2)=\lambda_1(P)^2$. Therefore,
$$
\vx_u^\intercal \left(\sum_{v \in N(u) \cup \{u\}}{\vx_v \vx_v^\intercal} \right)^2 \vx_u 
=  \vx_u^\intercal P^2 \vx_u 
\leq \lambda_1\left( P^2 \right)
= \lambda_1(P)^2
= O_H(1)
$$
for all $u \in V(G)$ and so we conclude that $\vartheta\left(\bar{G}\right)^3 \leq \sum_{u \in V(G)}{O_H(1)} = O_H(n)$. Since $G$ was arbitrary, we have established the first claim $\lambda(n, H) = O_H\left(n^{1/3}\right)$, from which the second claim $\mu(m, H) = O_H\left( m^{1/5}\right) $ follows via Lemma \ref{lemma:vertex to edge}.
\end{proof}

Note that Theorem \ref{thm:vxplusforest} follows from Theorem \ref{thm:H lambda} and Theorem \ref{thm:MaxCut bound with theta}.

\section{Concluding remarks}

In this paper we showed that if $G$ is a graph with $m$ edges, then
\begin{equation}
    \sp(G)\geq \frac{1}{\pi}\cdot \frac{m}{\vartheta(\bar{G})-1}. \label{eqn:lower bound for sp}
\end{equation}

This inequality further motivates the study of how large $\vartheta(\bar{G})$ can be for an $H$-free graph $G$. Regarding the case where $H$ is a clique, Alon and Szegedy \cite{ASz99} proved that for each $\varepsilon>0$ there exists some $t$ such that for any sufficiently large $m$ there are $K_t$-free graphs $G$ with $m$ edges satisfying $\vartheta(\bar{G})\geq m^{1/2-\varepsilon}$. This shows that one cannot use (\ref{eqn:lower bound for sp}) to prove that for all~$t$, $\sp(m,K_t)=\Omega_t(m^{\alpha})$ holds for some absolute constant $\alpha>1/2$. Balla \cite{Bal23} generalized the construction of Alon and Szegedy and proved that for each $\varepsilon>0$ there exists some $t$ such that for any sufficiently large $m$ there are $K_{t,t}$-free graphs $G$ with $m$ edges satisfying 
$\vartheta(\bar{G})\geq m^{1/2-\varepsilon}$. Moreover, it was shown in \cite{Bal23} that both these constructions have also 
$\chi_{\textrm{vec}}(G)\geq m^{1/2-\varepsilon}$. This implies that even our stronger result, Theorem \ref{thm:MaxCut bound with vecchrom}, cannot be used to prove that $\sp(m,K_t)=\Omega_t(m^{\alpha})$ or $\sp(m,K_{t,t})=\Omega_t(m^{\alpha})$ for some absolute constant $\alpha>1/2$ and all $t$.

It would be interesting to determine under what conditions the bound (\ref{eqn:lower bound for sp}) is tight (up to a constant factor). A sufficient condition is that $G$ is both vertex-transitive and edge-transitive. Indeed, it was shown in \cite{Lov79} that for any $n$-vertex graph $G$, if $G$ is vertex-transitive, then $\vartheta(G)\vartheta(\bar{G})=n$, and that if $G$ is edge-transitive, then $\vartheta(G)=-\frac{n\lambda_n}{\lambda_1-\lambda_n}$, where $\lambda_1$ and $\lambda_n$ are the largest and smallest eigenvalue of the adjacency matrix of $G$, respectively. Noting that $\lambda_1-\lambda_n=\Theta(\lambda_1)=\Theta(m/n)$ (as $G$ is regular), it follows that the right hand side of (\ref{eqn:lower bound for sp}) has order of magnitude $-n\lambda_n$, and hence $\sp(G)=\Omega(-n\lambda_n)$. On the other hand, it is well known \cite{DP93} that $\sp(G)=O(-n\lambda_n)$ holds for any $n$-vertex graph $G$, so (\ref{eqn:lower bound for sp}) is indeed tight up to a constant factor when $G$ is both vertex-transitive and edge-transitive.

While our lower bound (\ref{eqn:lower bound for sp}) for $\sp(G)$ is not tight in general, it is known, by a variant of Grothendieck's inequality (see \cite{Meg01} or \cite{CW04}), that the optimal value for the SDP relaxation of $\sp(G)$, namely the lower bound given in our Corollary \ref{cor:SDP}, is within a $\log n$ factor of the true value of $\sp(G)$. On a related note, we point out that in \cite{AMMN06}, $\vartheta(\bar{G})$ was used to bound the Grothendieck constant of $G$.

In \cite{S79}, Schrijver defined a variant of the theta function $\vartheta'(G)$, which also has several equivalent reformulations. Just like how the strict vector chromatic number of $\bar{G}$ is equivalent to the Lov\'asz theta function $\vartheta(G)$, the vector chromatic number $\chi_{\textrm{vec}}(\bar{G})$ is equivalent to Schrijver's theta function $\vartheta'(G)$, a fact which seems to have been first observed by Szegedy \cite{S94}. In view of Theorem \ref{thm:MaxCut bound with vecchrom}, it is natural to ask whether our arguments for bounding $\vartheta\left(\bar{G}\right)$ can be improved for $\vartheta'\left(\bar{G}\right)$. However, these arguments rely on inequalities involving the squares of the inner products $\inprod{\vx_u}{\vx_v}$ and therefore, it is not clear how one would make use of the additional assumption that $\inprod{\vx_u}{\vx_v} \geq 0$. Nonetheless, we note that Best (see \cite{S79}) constructed a graph satisfying $\vartheta'(G) < \vartheta(G)$, so the vector chromatic number sometimes provides a better bound for the surplus of a graph than $\vartheta\left(\bar{G}\right)$. It would therefore be interesting to determine how large the gap between $\vartheta'$ and $\vartheta$ can be.

\paragraph{Acknowledgements.}
We thank the referees for their careful review and valuable suggestions which improved the presentation.

\bibliographystyle{abbrv}
\bibliography{bibliography}

\end{document}